\documentclass{amsart}
\usepackage{amsfonts}
\usepackage{amsmath, amsthm}

\numberwithin{equation}{section}

\newtheorem{Theorem}{Theorem}[section]
\newtheorem{Lemma}[Theorem]{Lemma}

\newcommand{\cC}{\mathcal{C}}
\renewcommand{\O}{\mathcal{O}}
\newcommand{\R}{\mathbb{R}}
\newcommand{\cR}{\mathcal{R}}
\renewcommand{\S}{\mathbb{S}}
\newcommand{\cS}{\mathcal{S}}
\newcommand{\del}{\partial}
\renewcommand{\div}{\operatorname{div}}
\renewcommand{\phi}{\varphi}
\newcommand{\grad}{\nabla}
\renewcommand{\epsilon}{\varepsilon}
\newcommand{\Lip}{\mathrm{Lip}}
\renewcommand{\[}{\begin{equation}}
\renewcommand{\]}{\end{equation}}

\begin{document}

\title{Simons' cone and equivariant maximization of the first $p$-Laplace eigenvalue}
\author{Sinan Ariturk}
\address{Pontif\'icia Universidade Cat\'olica do Rio de Janeiro}
\email{ariturk@mat.puc-rio.br}

\begin{abstract}
We consider an optimization problem for the first Dirichlet eigenvalue of the $p$-Laplacian on a hypersurface in $\R^{2n}$, with $n \ge 2$.
If $p \ge 2n-1$, then among hypersurfaces in $\R^{2n}$ which are $O(n) \times O(n)$-invariant and have one fixed boundary component, there is a surface which maximizes the first Dirichlet eigenvalue of the $p$-Laplacian.
This surface is either Simons' cone or a $C^1$ hypersurface, depending on $p$ and $n$.
If $n$ is fixed and $p$ is large, then the maximizing surface is not Simons' cone.
If $p=2$ and $n \le 5$, then Simons' cone does not maximize the first eigenvalue.
\end{abstract}

\maketitle

\section{Introduction}

In this article we consider an optimization problem for the first Dirichlet eigenvalue of the $p$-Laplacian.
This problem is motivated by Simons' cone and by the Faber-Krahn inequality.
Simons' cone was the first example of a singular area minimizing cone.
Almgren \cite{Al} showed that the only area minimizing hypercones in $\R^4$ are hyperplanes.
Simons \cite{JSi} extended this to higher dimensions up to $\R^7$ and established the existence of a singular stable minimal hypercone in $\R^8$, given by
\[
	\bigg\{ (x_1, \ldots , x_4, y_1, \ldots, y_4) \in \R^8 : x_1^2 + \ldots + x_4^2 = y_1^2 + \ldots + y_4^2 \le 1 \bigg\}
\]
Bombieri, De Giorgi, and Giusti \cite{BDGG} showed that Simons' cone is area minimizing.
That is, Simons' cone has less volume than any other hypersurface in $\R^8$ with the same boundary.
Lawson \cite{La} and Simoes \cite{PSi} gave more examples of area minimizing hypercones.

The Faber-Krahn inequality states that among domains in $\R^n$ with fixed volume, the ball minimizes the first Dirichlet eigenvalue of the $p$-Laplacian for every $1<p<\infty$.
The $p$-Laplacian $\Delta_p$ is defined by
\[
	\Delta_p f = \div \Big(|\grad f|^{p-2} \grad f \Big)
\]
The Dirichlet eigenvalues of the $p$-Laplacian on a smoothly bounded domain $\Omega$ in $\R^n$ are the numbers $\lambda$ such that the equation $-\Delta_p \phi = \lambda |\phi|^{p-2} \phi$ admits a weak solution in $W_0^{1,p}(\Omega)$.
The $p$-Laplacian admits a smallest eigenvalue, denoted $\lambda_{1,p}(\Omega)$.
Lindqvist \cite{Li} showed this eigenvalue is simple on a connected domain, meaning the corresponding eigenfunction is unique up to normalization.
If $\Lip_0(\Omega)$ is the set of Lipschitz functions $f:\Omega \to \R$ which vanish on the boundary of $\Omega$, then $\lambda_{1,p}(\Omega)$ can be characterized variationally by
\[
	\lambda_{1,p}(\Omega) = \inf \bigg\{ \frac{ \int_\Omega | \grad f |^p}{\int_\Omega |f|^p} : f \in \Lip_0(\Omega) \bigg\}
\]
This characterization and the Polya-Szego inequality \cite{PS} imply the Faber-Krahn inequality.
Moreover, Brothers and Ziemer \cite{BZ} proved a uniqueness result.
In particular, the ball is the only minimizer with smooth boundary.

We consider a similar optimization problem for the first Dirichlet eigenvalue of the $p$-Laplacian on a hypersurface in $\R^{2n}$, for $n \ge 2$.
Let $G=O(n) \times O(n)$ and consider the usual action of $G$ on $\R^{2n}$.
Fix an orbit $\O$ of dimension $2n-2$.
Let $\cS$ be the set of all $C^1$ immersed $G$-invariant hypersurfaces in $\R^{2n}$ with one boundary component, given by $\O$.
For a hypersurface $\Sigma$ in $\cS$, the immersion of $\Sigma$ into $\R^{2n}$ induces a continuous Riemannian metric on $\Sigma$.
Let $\Lip_0(\Sigma)$ denote the set of Lipschitz functions $f: \Sigma \to \R$ which vanish on $\O$, and let $dV$ be the Riemannian measure on $\Sigma$.
Let $\lambda_{1,p}(\Sigma)$ denote the first Dirichlet eigenvalue of the $p$-Laplacian, which is given by
\[
\label{rayleigh}
	\lambda_{1,p}(\Sigma) = \inf \bigg\{ \frac{ \int_\Sigma | \grad f |^p \,dV} { \int_\Sigma |f|^p \,dV} : f \in \Lip_0(\Sigma) \bigg\}
\]
If $\O$ is the product of two spheres of the same radius $R$, then let $\Gamma$ be Simons' cone, defined by
\[
	\Gamma = \bigg\{ (x_1, \ldots , x_n, y_1, \ldots, y_n) \in \R^{2n} : x_1^2 + \ldots + x_n^2 = y_1^2 + \ldots + y_n^2 \le R^2 \bigg\}
\]
Let $\Lip_0(\Gamma)$ denote the set of Lipschitz functions $f: \Gamma \to \R$ which vanish on $\O$.
Note that $\Gamma \setminus \{ 0 \}$ is a smooth immersed hypersurface.
This immersion induces a Riemannian metric on $\Gamma \setminus \{ 0 \}$.
Let $dV$ be the Riemannian measure.
Then define
\[
\label{rayleigh}
	\lambda_{1,p}(\Gamma) = \inf \bigg\{ \frac{ \int_{\Gamma \setminus \{ 0 \}} | \grad f |^p \,dV} { \int_{\Gamma \setminus \{ 0 \}} |f|^p \,dV} : f \in \Lip_0(\Gamma) \bigg\}
\]
The following theorem states that if $p \ge 2n-1$, then there is a hypersurface which maximizes the eigenvalue $\lambda_{1,p}$.
This hypersurface is either Simons' cone or a $C^1$ hypersurface in $\cS$.

\begin{Theorem}
\label{thmex}
Fix $n \ge 2$ and $p \ge 2n-1$.
If $\O$ is the product of two spheres of different radii, then there is a $C^1$ embedded surface $\Sigma_p^*$ in $\cS$ such that
\[
	\lambda_{1,p}(\Sigma_p^*) = \sup \Big\{ \lambda_{1,p}(\Sigma) : \Sigma \in \cS \Big\}
\]
If $\O$ is the product of two spheres of the same radius and if
\[
\label{thmhyp}
	\lambda_{1,p}(\Gamma) < \sup \Big\{ \lambda_{1,p}(\Sigma) : \Sigma \in \cS \Big\}
\]
then there is a $C^1$ embedded surface $\Sigma_p^*$ in $\cS$ such that
\[
	\lambda_{1,p}(\Sigma_p^*) = \sup \Big\{ \lambda_{1,p}(\Sigma) : \Sigma \in \cS \Big\}
\]
\end{Theorem}
 
A natural problem motivated by Theorem \ref{thmex} is to determine if \eqref{thmhyp} holds, i.e. if Simons' cone maximizes $\lambda_{1,p}$.
For fixed $n$ and large $p$, we show that Simons' cone does not maximize $\lambda_{1,p}$.
For the case $p=2$ and $n \le 5$, we also show that Simons' cone does not maximize $\lambda_{1,2}$.

\begin{Theorem}
\label{thmnocone}
Assume $\O$ is the product of two spheres of the same radius.
For each $n$, there is a value $p_n$ such that if $p \ge p_n$, then
\[
\label{thmbigp}
	\lambda_{1,p}(\Gamma) < \sup \Big\{ \lambda_{1,p}(\Sigma) : \Sigma \in \cS \Big\}
\]
If $n \le 5$ and $p=2$, then
\[
\label{thmlb}
	\lambda_{1,2}(\Gamma) < \sup \Big\{ \lambda_{1,2}(\Sigma) : \Sigma \in \cS \Big\}
\]
\end{Theorem}

In particular, for the cases $n=4$ and $n=5$, Simons' cone is area minimizing, but does not maximize the eigenvalue $\lambda_{1,2}$.
This is in contrast to the inverse relationship that the eigenvalue and the volume of a domain often exhibit.
More accurately, the eigenvalues of a domain $\Omega$ in $\R^n$, are inversely related to the Cheeger constant $h(\Omega)$, which is defined by
\[
	h(\Omega) = \inf \bigg\{ \frac{|\del U|}{|U|} : U \subset \Omega \bigg\}
\]
Here $U$ is a smoothly bounded open subset of $\Omega$, and $|\del U|$ is the $(n-1)$-dimensional volume of $\del U$, while $|U|$ is the $n$-dimensional volume of $U$.
Cheeger's inequality states that
\[
\label{cheeger}
	\lambda_{1,p}(\Omega) \ge \bigg( \frac{h(\Omega)}{p} \bigg)^p
\]
Cheeger \cite{Ch} first proved this inequality for the case $p=2$.
Lefton and Wei \cite{LW}, Matei \cite{Ma}, and Takeuchi \cite{Ta} extended this inequality to the case $1<p<\infty$.
Moreover, Kawohl and Fridman \cite{KF} showed that
\[
	\lim_{p \to 1} \lambda_{1,p}(\Omega) = h(\Omega)
\]
We observe that the relationship between the eigenvalue $\lambda_{1,p}$ and the Cheeger constant is strongest for small $p$.
For large $p$, the eigenvalue $\lambda_{1,p}$ is more strongly related to the inradius of $\Omega$, denoted $\operatorname{inrad}(\Omega)$.
Juutinen, Lindqvist, and Manfredi \cite{JLM} proved that
\[
\label{eqjlm}
	\lim_{p \to \infty} \Big( \lambda_{1,p}(\Omega) \Big)^{1/p} = \frac{1}{\operatorname{inrad}(\Omega)}
\]
Moreover, Poliquin \cite{Po} showed that for each $p > n$, there is a constant $C_{n,p}$, independent of $\Omega$, such that
\[
\label{eqpo}
	\Big( \lambda_{1,p}(\Omega) \Big)^{1/p} \ge \frac{C_{n,p}}{\operatorname{inrad}(\Omega)}
\]
In light of \eqref{eqjlm} and \eqref{eqpo}, it is not surprising that Simons' cone does not maximize $\lambda_{1,p}$ for large $p$.
We remark that Grosjean \cite{Gr} established a result similar to \eqref{eqjlm} on a compact Riemannian manifold, with the inradius replaced by half the diameter of the manifold.
Valtorta \cite{Va} and Naber and Valtorta \cite{NV} obtained lower bounds for the first eigenvalue of the $p$-Laplacian in terms of the diameter on a compact Riemannian manifold.

A similar problem to the one described in Theorem \ref{thmex} is to maximize the first Dirichlet eigenvalue among surfaces of revolution in $\R^3$ with one fixed boundary component.
This problem has been studied for the case $p=2$.
It follows from a result of Abreu and Freitas \cite{AF} that the disc maximizes the first Dirichlet eigenvalue.
In fact the disc maximizes all of the Dirichlet eigenvalues \cite{Ar1}.
Moreover, it follows from a result of Colbois, Dryden, and El Soufi \cite{CDES} that a flat $n$-dimensional ball in $\R^{n+1}$ maximizes the first Dirichlet eigenvalue among $O(n)$-invariant hypersurfaces in $\R^{n+1}$ with the same boundary.
Among surfaces of revolution in $\R^3$ with two fixed boundary components, there is a smooth surface which maximizes the first Dirichlet eigenvalue \cite{Ar2}. 

The argument we use to prove Theorem \ref{thmex} is a development of the argument used in \cite{Ar2} to maximize Laplace eigenvalues on surfaces of revolution in $\R^3$.
For the case where $p$ is large, the proof of Theorem \ref{thmnocone} is a simple application of \eqref{eqjlm}.
For the case where $p=2$, we use a variational argument.
In the next section, we reformulate Theorem \ref{thmex} and Theorem \ref{thmnocone} as statements about curves in the orbit space $\R^{2n}/G$.
In the third section, we prove a low regularity version of Theorem \ref{thmex}.
In the fourth section, we complete the proof of Theorem \ref{thmex}.
In the fifth section, we prove Theorem \ref{thmnocone}.

\section{Reformulation}

In this section, we reformulate Theorem \ref{thmex} and Theorem \ref{thmnocone} as statements about curves in the orbit space $\R^{2n}/G$.
Identify $\R^{2n}/G$ with a quarter plane
\[
	\R^{2n}/G = \bigg\{ (x,y) \in \R^2 : x \ge 0, y \ge 0 \bigg\}
\]
A point $(x,y)$ is identified with the orbit
\[
	\Big\{ (x_1, \ldots, x_n, y_1, \ldots, y_n) \in \R^{2n} : x_1^2 + \ldots + x_n^2 = x^2, y_1^2 + \ldots + y_n^2 = y^2 \Big\}
\]
Let $g$ be the orbital distance metric on $\R^{2n}/G$, i.e. $g = dx^2 + dy^2$.
Define a function $F: \R^{2n}/G \to \R$ which maps an orbit to its $(2n-2)$-dimensional volume in $\R^{2n}$.
There is a constant $c_n$ such that
\[
	F(x,y) = c_n \cdot x^{n-1} y^{n-1}
\]
Let $(x_0,y_0)$ be the coordinates of the orbit $\O$.
By symmetry, we may assume that
\[
\label{x0y0}
	x_0 \ge y_0 > 0
\]

For a $C^1$ curve $\alpha:[0,1] \to \R^{2n}/G$, let $L_g(\alpha)$ be the length of $\alpha$ with respect to $g$.
Let $\cC$ be the set of $C^1$ curves $\alpha: [0,1] \to \R^{2n}/G$ which satisfy the following properties.
First $\alpha(0)=(x_0,y_0)$ and $\alpha(1)$ is in the boundary of $\R^{2n}/G$.
Second $\alpha(t)$ is in the interior of $\R^{2n}/G$ for every $t$ in $[0,1)$.
Third $|\alpha'(t)|_g=L_g(\alpha)$ for every $t$ in $[0,1]$.
Fourth $\alpha$ intersects the boundary of $\R^{2n}/G$ away from the origin, and the intersection is orthogonal.
If $\alpha$ is a curve in $\cC$, let $F_\alpha = F \circ \alpha$.
Let $\Lip_0([0,1])$ be the set of Lipschitz functions $w:[0,1] \to \R$ which vanish at zero.
Then define
\[
\label{ccrq}
	\lambda_{1,p}(\alpha) = \inf \left \{ \frac{ \int_0^1 \frac{|w'|^p F_\alpha}{|\alpha '|_g^{p-1}} \,dt}{ \int_0^1 |w|^p F_\alpha | \alpha' |_g \,dt} : w \in \Lip_0([0,1]) \right \}
\]
For a function $w$ in $\Lip_0([0,1])$, the Rayleigh quotient of $w$ is
\[
	\frac{ \int_0^1 \frac{|w'|^p F_\alpha}{|\alpha '|_g^{p-1}} \,dt}{ \int_0^1 |w|^p F_\alpha | \alpha' |_g \,dt}
\]
Note that if $\alpha$ is in $\cC$, then there is a corresponding surface $\Sigma$ in $\cS$ such that $\alpha$ parametrizes the projection of $\Sigma$ in $\R^{2n}/G$.
Moreover $\lambda_{1,p}(\Sigma) = \lambda_{1,p}(\alpha)$, because the first eigenfunction on $\Sigma$ is $G$-invariant.
Furthermore, if $\Sigma$ is a surface in $\cS$ and $\lambda_{1,p}(\Sigma)$ is non-zero, then $\Sigma$ is connected and there is a curve $\alpha$ in $\cC$ corresponding to $\Sigma$.
In particular,
\[
	\sup \Big\{ \lambda_{1,p}(\Sigma) : \Sigma \in \cS \Big\} = \sup \Big\{ \lambda_{1,p}(\alpha) : \alpha \in \cC \Big\}
\]

If $x_0=y_0=R$, then define a curve $\sigma: [0,1] \to \R^{2n}/G$ by
\[
	\sigma(t)= (1-t) \cdot (R, R)
\]
Let $F_\sigma = F \circ \sigma$ and define
\[
	\lambda_{1,p}(\sigma) = \inf \left \{ \frac{ \int_0^1 \frac{|w'|^p F_\sigma}{|\sigma '|_g^{p-1}} \,dt}{ \int_0^1 |w|^p F_\sigma | \sigma' |_g \,dt} : w \in \Lip_0([0,1]) \right \}
\]
This curve corresponds to Simons' cone $\Gamma$ in $\R^{2n}/G$, and $\lambda_{1,p}(\Gamma) = \lambda_{1,p}(\sigma)$.

\begin{Lemma}
\label{exreg}
Fix $n \ge 2$ and $p \ge 2n-1$.
If $x_0 \neq y_0$ then there is a simple curve $\alpha$ in $\cC$ such that
\[
	\lambda_{1,p}(\alpha) = \sup \Big\{ \lambda_{1,p}(\beta) : \beta \in \cC \Big\}
\]
If $x_0=y_0$ and if
\[
	\lambda_{1,p}(\sigma) < \sup \Big\{ \lambda_{1,p}(\beta) : \beta \in \cC \Big\}
\]
then there is a simple curve $\alpha$ in $\cC$ such that
\[
	\lambda_{1,p}(\alpha) = \sup \Big\{ \lambda_{1,p}(\beta) : \beta \in \cC \Big\}
\]
\end{Lemma}

Lemma \ref{exreg} immediately yields Theorem \ref{thmex}.
In the third section, we prove low regularity versions of Lemma \ref{exreg}, replacing $\cC$ with larger spaces of curves.
First, we identify points on the boundary of $\R^{2n}/G$ to obtain a quotient space $Q$, and we prove existence in a space $\cR_h$ of curves in $Q$ which correspond to hypersurfaces in $\R^{2n}$ of finite volume.
Then we obtain a maximizing curve in a space $\cR_h^+$ of curves in $\R^{2n}/G$ which project to curves in $\cR_h$.
Finally, we prove existence in a space $\cR_g^*$ of curves in $\R^{2n}/G$ which have finite length and intersect the boundary transversally, away from the singular point.
In the fourth section, we complete the proof of Lemma \ref{exreg} by showing that a maximizing curve in $\cR_g^*$ must be in $\cC$.

\begin{Lemma}
\label{lemnocone}
Assume $x_0=y_0$.
For each $n$, there is a value $p_n$ such that if $p \ge p_n$, then
\[
	\lambda_{1,p}(\sigma) < \sup \Big\{ \lambda_{1,p}(\alpha) : \alpha \in \cC \Big\}
\]
If $n \le 5$ and $p=2$, then
\[
	\lambda_{1,2}(\sigma) < \sup \Big\{ \lambda_{1,2}(\alpha) : \alpha \in \cC \Big\}
\]
\end{Lemma}

Lemma \ref{lemnocone} immediately yields Theorem \ref{thmnocone}.
We prove Lemma \ref{lemnocone} in the fifth section of the article.
For the case where $p$ is large, the proof is a simple application of \eqref{eqjlm}.
For the case where $p=2$, we use a variational argument.

\section{Existence}

In this section we prove a low regularity version of Lemma \ref{exreg}.
We first extend the defintion of $\lambda_{1,p}$ to low regularity curves.
Define a Riemannian metric $h$ on the interior of $\R^{2n}/G$ by
\[
	h = F^2 \cdot g = c_n^2 \cdot x^{2n-2} y^{2n-2}  \Big( dx^2 + dy^2 \Big)
\]
The length of a curve in $\R^{2n}/G$ with respect to $h$ is the $(2n-1)$-dimensional volume of the corresponding $G$-invariant hypersurface in $\R^{2n}$.
Define an equivalence relation on $\R^{2n}/G$ such that each point in the interior is only equivalent to itself, and any two points on the boundary are equivalent.
Let $Q$ be the quotient space of $\R^{2n}/G$ with respect to this equivalence relation.
Let $Q_0$ be the image of the interior of $\R^{2n}/G$ under the quotient map.
Let $Q_B$ denote the remaining point in $Q$ which is the image of the boundary of $\R^{2n}/G$.
Then $Q = Q_0 \cup \{ Q_B \}$.
We view $Q$ as a metric space, with distance function induced by $h$.
The function $F: \R^{2n}/G \to \R$ induces a function on $Q$, which we also denote by $F$.
Let $\alpha: [c,d] \to Q$ be a Lipschitz curve such that $\alpha(t) \neq Q_B$ for all $t$ in $[c,d)$ and $\alpha(d)=Q_B$.
Let $\Lip_0([c,d])$ be the set of Lipschitz functions $w:[c,d] \to \R$ which vanish at $c$.
Let $F_\alpha = F \circ \alpha$ and define
\[
\label{rhrq}
	\lambda_{1,p}(\alpha) = \inf \left \{ \frac{ \int_c^d \frac{| w' |^p F_\alpha^p}{|\alpha'|_h^{p-1}} \,dt}{\int_c^d | w |^p |\alpha'|_h \,dt} : w \in \Lip_0([c,d]) \right \}
\]
If the integrand in the numerator takes the form $0/0$ at some point in $[c,d]$, then we interpret the integrand as being equal to zero at this point.
If the Rayleigh quotient takes the form $0/0$, then we interpret the Rayleigh quotient as being infinite.
Let $\cR_h$ be the set of Lipschitz curves $\alpha: [0,1] \to Q$ such that $\alpha(0)=(x_0,y_0)$ and $\alpha(1)=Q_B$ and $\alpha(t)$ is in $Q_0$ for every $t$ in $[0,1)$.
Note that a curve in $\cC$ can be identified with a curve in $\cR_h$, by composing with the quotient map $\R^{2n}/G \to Q$.
Making this identification, the definitions \eqref{ccrq} and \eqref{rhrq} are the same.

For a Lipschitz curve $\gamma:[c,d] \to Q$, let $L_h(\gamma)$ denote the length of $\gamma$.
In the following lemma, we prove that reparametrizing a curve by arc length with respect to $h$ does not decrease the eigenvalue.

\begin{Lemma}
\label{harc}
Let $\gamma: [c,d] \to Q$ be a Lipschitz curve such that $\gamma(c) = (x_0,y_0)$ and $\gamma(d) = Q_B$.
Assume that $\gamma(t) \neq Q_B$ for all $t$ in $[c,d)$.
Define $\ell_h:[c,d] \to [0,1]$ by
\[
	\ell_h(t) = \frac{1}{L_h(\gamma)} \int_c^t | \gamma'(u)|_h \,du
\]
There is a curve $\beta$ in $\cR_h$ such that $\beta(\ell_h(t)) = \gamma(t)$ for all $t$ in $[c,d]$.
Moreover $| \beta'(t) |_h = L_h(\beta)$ for almost every $t$ in $[0,1]$, and $L_h(\beta)=L_h(\gamma)$.
Furthermore $\lambda_{1,p}(\beta) \ge \lambda_{1,p}(\gamma)$ for every $p \ge 2n-1$.
\end{Lemma}

\begin{proof}
Define $\eta:[0,1] \to [c,d]$ by
\[
	\eta(s) = \min \Big\{ t \in [c,d] : \ell_h(t)=s \Big\}
\]
Note that $\eta$ may not be continuous, but $\beta = \gamma \circ \eta$ is in $\cR_h$, and $\beta(\ell_h(t)) = \gamma(t)$ for all $t$ in $[c,d]$.
Also $| \beta'(t) |_h = L_h(\gamma)$ for almost every $t$ in $[0,1]$, so $L_h(\beta)=L_h(\gamma)$.
Let $F_\gamma = F \circ \gamma$ and $F_\beta = F \circ \beta$.
Let $w$ be a function in $\Lip_0([0,1])$.
Define $v = w \circ \ell_h$.
Then $v$ is in $\Lip_0([c,d])$, and changing variables yields
\[
	\lambda_{1,p}(\gamma) \le \frac{ \int_c^d \frac{| v' |^p F_\gamma^p}{|\gamma'|_h^{p-1}} \,dt}{\int_c^d | v |^p |\gamma'|_h \,dt} = \frac{ \int_0^1 \frac{| w' |^p F_\beta^p}{|\beta'|_h^{p-1}} \,dt}{\int_0^1 | w |^p |\beta'|_h \,dt}
\]
Since $w$ is arbitrary, this implies that $\lambda_{1,p}(\gamma) \le \lambda_{1,p}(\beta)$.
\end{proof}

In the following lemma, we bound the length $L_h(\gamma)$ of a curve $\gamma$ in $\cR_h$ in terms of the eigenvalue $\lambda_{1,p}(\gamma)$.

\begin{Lemma}
\label{hLbound}
Fix $p \ge 2n-1$.
There is a constant $C_p$ such that for any $\gamma$ in $\cR_h$,
\[
	L_h(\gamma) \le \frac{C_p}{\lambda_{1,p}(\gamma)}
\]
\end{Lemma}

\begin{proof}
Let $\beta$ be the reparametrization given by Lemma \ref{harc} so that $\lambda_{1,p}(\beta) \ge \lambda_{1,p}(\gamma)$ and $| \beta'(t) |_h = L_h(\gamma)$ for almost every $t$ in $[0,1]$.
Let $r>0$ be a small number.
Define $w: [0,1] \to \R$ by
\[
	w(t) =
	\begin{cases}
		\frac{L_h(\gamma)}{r} \cdot t & 0 \le t \le \frac{r}{L_h(\gamma)} \\
		1 & \frac{r}{L_h(\gamma)} \le t \le 1 \\
	\end{cases}
\]
Let $F_\beta = F \circ \beta$.
Then there is a constant $C_p$, which is independent of $\gamma$ and $\beta$, such that
\[
\label{hLbound1}
	\lambda_{1,p}(\gamma) \le \lambda_{1,p}(\beta)
	\le \frac{ \int_0^{\frac{r}{L_h(\gamma)}} \frac{| w' |^p F_\beta^p}{|\beta'|_h^{p-1}} \,dt}{\int_{\frac{r}{L_h(\gamma)}}^1 | w |^p |\beta'|_h \,dt} \le \frac{C_p}{L_h(\gamma)}
\]
To verify the last inequality, note that if $r$ is small, then there is a constant $C_p'$ such that $F_\beta^p \le C_p'$ over $[0,r/L_h(\gamma)]$.
Therefore,
\[
	\frac{ \int_0^{\frac{r}{L_h(\gamma)}} \frac{| w' |^p F_\beta^p}{|\beta'|_h^{p-1}} \,dt}{\int_{\frac{r}{L_h(\gamma)}}^1 | w |^p |\beta'|_h \,dt}
		\le \frac{C_p'}{L_h(\gamma) r^{p-1}(1-\frac{r}{L_h(\gamma)})}
\]
If $r$ is small, this establishes \eqref{hLbound1}.
\end{proof}

The purpose of the next lemma is to show that there is an eigenvalue maximizing sequence of curves in $\cR_h$ whose images are contained in a fixed compact subset of $Q$.
Let $\rho_0=\sqrt{x_0^2+y_0^2}$ and define
\[
	K = \Big\{ (x,y) \in \R^{2n}/G : x^2+y^2 \le \rho_0^2 \Big\}
\]
Let $Q_K$ be the image of $K$ under the quotient map $\R^{2n}/G \to Q$.

\begin{Lemma}
\label{inversion}
Let $\alpha$ be a curve in $\cR_h$.
There is a curve $\beta$ in $\cR_h$ such that $\beta(t)$ is in $Q_K$ for all $t$ in $[0,1]$ and $\lambda_{1,p}(\beta) \ge \lambda_{1,p}(\alpha)$ for all $p \ge 2n-1$.
\end{Lemma}

\begin{proof}
There are functions $r_\alpha:[0,1) \to \R$ and $\theta_\alpha:[0,1) \to [0, \pi/2]$
such that for all $t$ in $[0,1)$,
\[
	\alpha(t) = \Big( r_\alpha(t) \cos \theta_\alpha(t), r_\alpha(t) \sin \theta_\alpha(t) \Big)
\]
Define a function $r_\beta:[0,1) \to [0,\rho_0]$ by
\[
	r_\beta(t) = \min \Big( \frac{\rho_0^2}{r_\alpha(t)}, r_\alpha(t) \Big)
\]
Then define a curve $\beta$ in $\cR_h$ so that $\beta(1)=Q_B$ and for all $t$ in $[0,1)$,
\[
	\beta(t) = \Big( r_\beta(t) \cos \theta_\alpha(t), r_\beta(t) \sin \theta_\alpha(t) \Big)
\]
Then $\beta$ is in $\cR_h$ and $\beta(t)$ is in $Q_K$ for all $t$ in $[0,1]$.
Let $F_\alpha=F \circ \alpha$ and $F_\beta=F \circ \beta$.
For all $p \ge 2n-1$ and for almost every $t$ in $[0,1]$,
\[
\label{inversionconf}
	\frac{(F_\alpha(t))^p}{|\alpha'(t)|_h^{p-1}} \le \frac{(F_\beta(t))^p}{|\beta'(t)|_h^{p-1}}
\]
To verify this, note that for all $t$ in $[0,1)$,
\[
	\frac{F_\alpha(t)}{F_\beta(t)} = \frac{r_\alpha^{2n-2}}{r_\beta^{2n-2}}
\]
Also, for almost every $t$ in $[0,1]$,
\[
	| \alpha'(t) |_h = \frac{r_\alpha^{2n-1}}{r_\beta^{2n-1}} \cdot | \beta'(t) |_h
\]
Therefore \eqref{inversionconf} follows, because $p \ge 2n-1$.
Also $| \alpha'(t)|_h \ge | \beta'(t)|_h$ for almost every $t$ in $[0,1]$.
Therefore $\lambda_{1,p}(\alpha) \le \lambda_{1,p}(\beta)$ for all $p \ge 2n-1$.
\end{proof}

We can now establish the existence of an eigenvalue maximizing curve in $\cR_h$.
For $p \ge 2n-1$, define
\[
	\Lambda_p = \sup \bigg\{ \lambda_{1,p}(\alpha) : \alpha \in \cR_h \bigg\}
\]

\begin{Lemma}
\label{hex}
Fix $p \ge 2n-1$.
There is a curve $\alpha$ in $\cR_h$ such that $\lambda_{1,p}(\alpha) = \Lambda_p$ and $\alpha(t)$ is in $Q_K$ for all $t$ in $[0,1]$.
\end{Lemma}

\begin{proof}
Let $\{ \gamma_j \}$ be a sequence in $\cR_h$ such that
\[
	\lim_{j \to \infty} \lambda_{1,p}(\gamma_j) = \Lambda_p
\]
By Lemma \ref{inversion}, we may assume that $\gamma_j(t)$ is in $Q_K$ for every $j$ and every $t$ in $[0,1]$.
Using Lemma \ref{harc}, we may assume that $| \gamma_j'(t) |_h = L_h(\gamma_j)$ for every $j$ and almost every $t$ in $[0,1]$.
By Lemma \ref{hLbound}, the lengths $L_h(\gamma_j)$ are uniformly bounded.
By passing to a subsequence, we may assume that the lengths $L_h(\gamma_j)$ converge to some positive number $\ell$.
The curves $\gamma_j$ are uniformly Lipschitz.
Therefore, by applying the Arzela-Ascoli theorem and passing to a subsequence, we may assume that the curves $\gamma_j$ converge uniformly to a Lipschitz curve $\gamma:[0,1] \to Q_K$.
Moreover $| \gamma'(t) |_h \le \ell$ for almost every $t$ in $[0,b]$.
For each $j$, define $F_j = F \circ \gamma_j$.
Also define $F_\gamma = F \circ \gamma$.
Define $b$ in $(0,1]$ by
\[
	b= \min \Big\{ t \in [0,1] : \gamma(t) = Q_B \Big\}
\]
Let $w$ be in $\Lip_0([0,b])$.
Define $v$ to be a function in $\Lip_0([0,1])$ which agrees with $w$ over $[0,b]$ and is constant over $[0,1]$.
Then
\[
	\Lambda_p = \lim_{j \to \infty} \lambda_{1,p}(\gamma_j)
		\le \liminf_{j \to \infty} \frac{ \int_0^1 \frac{| v' |^p F_j^p}{L_h(\gamma_j)^{p-1}} \,dt}{\int_0^1 | v |^p L_h(\gamma_j) \,dt}
		\le \liminf_{j \to \infty} \frac{ \int_0^b \frac{| w' |^p F_j^p}{L_h(\gamma_j)^{p-1}} \,dt}{\int_0^b | w |^p L_h(\gamma_j) \,dt}
\]
Moreover $F_j$ converges to $F_\gamma$ uniformly over $[0,b]$, because $F$ is continuous on $Q_K$.
Also $L_h(\gamma_j)$ converges to $\ell$, so
\[
	\lim_{j \to \infty} \frac{ \int_0^b \frac{| w' |^p F_j^p}{L_h(\gamma_j)^{p-1}} \,dt}{\int_0^b | w |^p L_h(\gamma_j) \,dt}
		= \frac{ \int_0^b \frac{| w' |^p F_\gamma^p}{\ell^{p-1}} \,dt}{\int_0^b | w |^p \ell \,dt} \\
		\le \frac{ \int_0^b \frac{| w' |^p F_\gamma^p}{|\gamma'|_h^{p-1}} \,dt}{\int_0^b | w |^p |\gamma'|_h \,dt}
\]
Therefore
\[
	\Lambda_p \le \frac{ \int_0^b \frac{| w' |^p F_\gamma^p}{|\gamma'|_h^{p-1}} \,dt}{\int_0^b | w |^p |\gamma'|_h \,dt}
\]
Since $w$ is arbitrary,
\[
	\Lambda_p \le \lambda_{1,p}\Big( \gamma \big|_{[0,b]} \Big)
\]
Let $\alpha$ be the reparametrization of $\gamma|_{[0,b]}$ given by Lemma \ref{harc}.
Then $\alpha$ is in $\cR_h$ and $\alpha(t)$ is in $Q_K$ for all $t$ in $[0,1]$.
Moreover
\[
	\lambda_{1,p}(\alpha) \ge \lambda_{1,p}(\gamma|_{[0,b]}) \ge \Lambda_p
\]
Therfore $\lambda_{1,p}(\alpha) = \Lambda_p$.
\end{proof}

Let $\cR_h^+$ be the set of continuous curves $\alpha:[0,1] \to \R^{2n}/G$ such that composition with the quotient map $\R^{2n}/G \to Q$ yields a curve in $\cR_h$.
We use \eqref{rhrq} to define $\lambda_{1,p}(\alpha)$ for $\alpha$ in $\cR_h^+$.
In the next lemma we establish existence of an eigenvalue maximizing curve in $\cR_h^+$.
We first introduce new coordinates functions on $\R^{2n}/G$.
Define $u: \R^{2n}/G \to \R$ and $v: \R^{2n}/G \to [0, \infty)$ by
\[
\label{ucoord}
	u(x,y) = \frac{1}{2} (x^2-y^2)
\]
and
\[
\label{vcoord}
	v(x,y) = xy
\]
These coordinates can be used to identify $\R^{2n}/G$ with a half-plane.
A feature of these coordinates is that the function $F$ can be expressed as $F= c_n \cdot v^{n-1}$.
Define a function $r= \sqrt{u^2+v^2}$.
Then the metric $h$ can be expressed as
\[
	h = \frac{c_n^2 \cdot v^{2n-2}}{2r} \Big( du^2 + dv^2 \Big)
\]
By \eqref{x0y0}, we have
\[
\label{u0}
	u(x_0,y_0) \ge 0
\]

\begin{Lemma}
\label{umono}
Fix $p \ge 2n-1$.
There is a curve $\alpha$ in $\cR_h^+$ such that $\lambda_{1,p}(\alpha) = \Lambda_p$.
Moreover $u \circ \alpha$ is monotonically increasing over $[0,1]$.
\end{Lemma}

\begin{proof}
By Lemma \ref{hex}, there is a curve $\beta$ in $\cR_h$ such that $\lambda_{1,p}(\beta) = \Lambda_p$ and $\beta(t)$ is in $Q_K$ for all $t$ in $[0,1]$.
Define functions $u_\beta:[0,1) \to \R$ and $v_\beta:[0,1) \to [0,\infty)$ by $u_\beta = u \circ \beta$ and $v_\beta=v \circ \beta$.
Note that $u_\beta$ is bounded over $[0,1)$.
Also $v(t)>0$ for all $t$ in $[0,1)$ and
\[
	\lim_{t \to 1} v_\beta(t) = 0
\]
In particular $v_\beta$ has a continuous extension to $[0,1]$.
Let $v_\alpha:[0,1] \to [0,\infty)$ be this extension.
Note that $u_\beta(0) \ge 0$ by \eqref{u0}, and define $u_\alpha : [0,1] \to \R$ by
\[
	u_\alpha(t) = \sup \bigg\{ |u_\beta(s)| : s \in [0,t) \bigg\}
\]
Then $u_\alpha$ is monotonically increasing, continuous, and bounded.
Define a continuous curve $\alpha:[0,1] \to \R^{2n}/G$ so that $u \circ \alpha = u_\alpha$ and $v \circ \alpha = v_\alpha$.
Then $\alpha$ is in $\cR_h^+$, and $u \circ \alpha$ is monotonically increasing over $[0,1]$.
Let $F_\alpha = F \circ \alpha$ and $F_\beta = F \circ \beta$.
Note that $F_\alpha=F_\beta$ over $[0,1]$, and $| \alpha'(t) |_h \le | \beta'(t) |_h$ for almost every $t$ in $[0,1]$.
Therefore $\lambda_{1,p}(\alpha) \ge \lambda_{1,p}(\beta)$, hence $\lambda_{1,p}(\alpha)=\Lambda_p$.
\end{proof}

In the following lemma, we establish existence of a maximizing curve $\alpha$ in $\cR_h^+$ such that $u \circ \alpha$ is monotonically increasing and $r \circ \alpha$ is monotonically decreasing.

\begin{Lemma}
\label{rumono}
Fix $p \ge 2n-1$.
There is a curve $\alpha$ in $\cR_h^+$ such that $\lambda_{1,p}(\alpha)=\Lambda_p$.
Moreover $u \circ \alpha$ is monotonically increasing over $[0,1]$ and $r \circ \alpha$ is monotonically decreasing over $[0,1]$.
\end{Lemma}

\begin{proof}
By Lemma \ref{umono}, there is a curve $\gamma$ in $\cR_h^+$ such that $\lambda_{1,p}(\gamma) = \Lambda_p$.
Moreover $u \circ \gamma$ is monotonically increasing over $[0,1]$.
Define $r_\gamma = r \circ \gamma$.
Note that $r_\gamma$ is non-vanishing over $[0,1)$.
There is a function $\theta_\gamma:[0,1] \to [0,\pi/2]$ such that
\[
	\Big( u \circ \gamma, v \circ \gamma \Big) = \Big( r_\gamma \cos \theta_\gamma, r_\gamma \sin \theta_\gamma \Big)
\]
Define a function $r_\beta:[0,1] \to (0,\infty)$ by
\[
	r_\beta(t) = \min \Big\{ r_\gamma(s) : s \in [0,t] \Big\}
\]
Define a curve $\beta:[0,1] \to \R^{2n}/G$ such that
\[
	\Big( u \circ \beta, v \circ \beta \Big) = \Big( r_\beta \cos \theta_\gamma, r_\beta \sin \theta_\gamma \Big)
\]
Note that $\beta$ is in $\cR_h^+$ and $r_\beta$ is monotonically decreasing over $[0,1]$.
Define a set $W$ by
\[
	W = \Big\{ t \in [0,1] : \beta(t) = \gamma(t) \Big\}
\]
The isolated points of $W$ are countable, so $\beta'(t)=\gamma'(t)$ for almost every $t$ in $W$.
Note that there are countably many disjoint intervals $(a_1,b_1), (a_2, b_2), \ldots$ such that
\[
	[0,1) \setminus W = \bigcup_{j} (a_j, b_j)
\]
Moreover $r_\beta$ is constant on each interval $(a_j, b_j)$.
For all $p \ge 2n-1$ and for almost every $t$ in $[0,1]$,
\[
\label{rumonoconf}
	\frac{(F_\beta(t))^p}{|\beta'(t)|_h^{p-1}} \ge \frac{(F_\gamma(t))^p}{|\gamma'(t)|_h^{p-1}}
\]
To verify this, note that for all $t$ in $[0,1)$,
\[
	\frac{F_\gamma(t)}{F_\beta(t)} = \frac{r_\gamma^{2n-2}}{r_\beta^{2n-2}}
\]
Also, for almost every $t$ in $[0,1]$,
\[
	| \gamma' |_h \ge \frac{r_\gamma^{2n-1}}{r_\beta^{2n-1}} \cdot | \beta' |_h
\]
Therefore \eqref{rumonoconf} follows, because $p \ge 2n-1$.
Also $| \beta'(t)|_h \le | \gamma'(t)|_h$ for almost every $t$ in $[0,1]$.
Therefore $\lambda_{1,p}(\beta) \ge \lambda_{1,p}(\gamma)$, so $\lambda_{1,p}(\beta)=\Lambda_p$.
Define $\theta_\alpha:[0,1] \to [0, \pi/2]$ by
\[
	\theta_\alpha(t) =
	\begin{cases}
		\theta_\gamma(t) & t \in W \\
		\min \Big\{ \theta_\gamma(s) : s \in [a_j,t] \Big\} & t \in (a_j,b_j) \\
	\end{cases}
\]
The main difficulty in proving $\theta_\alpha$ is Lipschitz continuous is to verify that for each point $b_j$,
\[
\label{bcont}
	\min \Big\{ \theta_\gamma(s) : s \in [a_j,b_j] \Big\} = \theta_\gamma(b_j)
\]
To prove this, fix $s$ in $[a_j, b_j]$
Then
\[
	r_\gamma(b_j) = r_\beta(b_j) = r_\beta(s) \le r_\gamma(s)
\]
Also $u \circ \gamma$ is monotonically increasing, so $u \circ \gamma(s) \le u \circ \gamma(b_j)$, i.e.
\[
	r_\gamma(s) \cos \theta_\gamma(s) \le r_\gamma(b_j) \cos \theta_\gamma(b_j)
\]
It follows that $\theta_\gamma(s) \ge \theta_\gamma(b_j)$, which proves \eqref{bcont}.

Define a curve $\alpha$ in $\cR_h^+$ such that
\[
	\Big( u \circ \alpha, v \circ \alpha \Big) = \Big( r_\beta \cos \theta_\alpha, r_\beta \sin \theta_\alpha \Big)
\]
Note that $r \circ \alpha = r_\beta$ is monotonically decreasing over $[0,1]$.
Additionally $u \circ \alpha$ is monotonically increasing over $W$, because $\alpha = \gamma$ over $W$.
Also $u \circ \alpha$ is monotonically increasing over each interval $(a_j,b_j)$, because $r_\beta$ is constant and $\theta_\alpha$ is monotonically decreasing over each of these intervals.
Therefore $u \circ \alpha$ is monotonically increasing over $[0,1]$.
In order to show that $\lambda_{1,p}(\alpha) \ge \lambda_{1,p}(\beta)$, define a set $Z$ by
\[
	Z = \Big\{ t \in [0,1] : \alpha(t) = \beta(t) \Big\}
\]
Note that $W$ is contained in $Z$, and there are countably many disjoint intervals $(c_1,d_1), (c_2, d_2), \ldots$ such that
\[
	[0,1) \setminus Z = \bigcup_{j} (c_j, d_j)
\]
Moreover $\theta_\alpha$ and $r_\beta$ are constant on each interval $(c_j, d_j)$.
That is, $\alpha$ is constant on each interval $(c_j,d_j)$.
Let $w$ be a function in $\Lip_0([0,1])$ such that
\[
	\frac{ \int_0^1 \frac{| w' |^p F_\alpha^p}{|\alpha'|_h^{p-1}} \,dt}{\int_0^1 | w |^p |\alpha'|_h \,dt} < \infty
\]
In particular $w$ is constant on each interval $(c_j,d_j)$.
Additionally, the isolated points of $Z$ are countable, so $\alpha'(t)=\beta'(t)$ for almost every $t$ in $Z$.
Therefore
\[
	\lambda_{1,p}(\beta) \le \frac{ \int_0^1 \frac{| w' |^p F_\beta^p}{|\beta'|_h^{p-1}} \,dt}{\int_0^1 | w |^p |\beta'|_h \,dt}
		\le \frac{ \int_Z \frac{| w' |^p F_\beta^p}{|\beta'|_h^{p-1}} \,dt}{\int_Z | w |^p |\beta'|_h \,dt}
		= \frac{ \int_0^1 \frac{| w' |^p F_\alpha^p}{|\alpha'|_h^{p-1}} \,dt}{\int_0^1 | w |^p |\alpha'|_h \,dt}
\]
Since $w$ is arbitrary, we have $\lambda_{1,p}(\alpha) \ge \lambda_{1,p}(\beta)$.
Therefore $\lambda_{1,p}(\alpha)=\Lambda_p$.
\end{proof}

Recall $g$ is the orbital distance metric on $\R^{2n}/G$, i.e. $g= dx^2 + dy^2$.
The metric $g$ can also be expressed as
\[
\label{guv}
	g = \frac{1}{2r} \Big( du^2 + dv^2 \Big)
\]
We view $\R^{2n}/G$ as a metric space, with distance function induced by $g$.
Let $\cR_g$ be the set of Lipschitz curves $\alpha:[0,1] \to \R^{2n}/G$ such that $\alpha(0)=(x_0,y_0)$ and $\alpha(1)$ is in the boundary of $\R^{2n}/G$ and $\alpha(t)$ in the interior of $\R^{2n}/G$ for every $t$ in $[0,1)$.
Note that $\cR_g$ is a subset of $\cR_h^+$.
If $\alpha$ is in $\cR_g$ and $F_\alpha=F \circ \alpha$, then
\[
\label{rgrq}
	\lambda_{1,p}(\alpha) = \inf \left \{ \frac{ \int_0^1 \frac{|w'|^p F_\alpha}{|\alpha '|_g^{p-1}} \,dt}{ \int_0^1 |w|^p F_\alpha | \alpha' |_g \,dt} : w \in \Lip_0([0,1]) \right \}
\]

The previous lemma can be used to establish existence of an eigenvalue maximizing curve in $\cR_h^+$ which has finite length with respect to $g$.
The following lemma shows that reparametrization then yields a curve in $\cR_g$.
The statement and proof are very similar to Lemma \ref{harc}, with the metric $g$ in place of the metric $h$.
For a curve $\alpha$ in $\cR_h^+$, let $L_g(\alpha)$ denote the length of $\alpha$ with respect to $g$.

\begin{Lemma}
\label{garc}
Let $\beta$ be a curve in $\cR_h^+$ and assume that $L_g(\beta)$ is finite.
Define $\ell_g:[0,1] \to [0,1]$ by
\[
	\ell_g(t) = \frac{1}{L_g(\beta)} \int_0^t | \beta'(u)|_g \,du
\]
There is a curve $\alpha$ in $\cR_g$ such that $\alpha(\ell_g(t))=\beta(t)$ for all $t$ in $[0,1]$, and $| \alpha'(t) |_g = L_g(\alpha)$ for almost every $t$ in $[0,1]$.
Also $\lambda_{1,p}(\alpha) \ge \lambda_{1,p}(\beta)$ for all $p \ge 2$.
\end{Lemma}

\begin{proof}
First note that $\beta$ is locally Lipschitz over $[0,1)$ and continuous over $[0,1]$.
Define $\eta:[0,1] \to \R$ by
\[
	\eta(s) = \min \Big\{ t \in [0,1] : \ell_g(t)=s \Big\}
\]
Define $\alpha:[0,1] \to \R^{2n}/G$ by $\alpha= \beta \circ \eta$.
Note that $\eta$ may not be continuous, but $\alpha$ is locally Lipschitz over $[0,1)$ and continuous over $[0,1]$.
Moreover $\alpha(\ell_g(t)) = \beta(t)$ for all $t$ in $[0,1]$.
Therefore $| \alpha'(t) |_g = L_g(\alpha)$ for almost every $t$ in $[0,1]$.
In particular $\alpha$ is in $\cR_g$.
Let $F_\beta = F \circ \beta$ and $F_\alpha = F \circ \alpha$.
Let $w$ be in $\Lip_0([0,1])$.
Define $v = w \circ \ell_g$.
Then $v$ is in $\Lip_0([0,1])$, and changing variables yields
\[
	\lambda_{1,p}(\beta) \le \frac{ \int_0^1 \frac{| v' |^p F_\beta}{|\beta'|_g^{p-1}} \,dt}{\int_0^1 | v |^p F_\beta |\beta'|_g \,dt}
		= \frac{ \int_0^1 \frac{| w' |^p F_\alpha}{|\alpha'|_g^{p-1}} \,dt}{\int_0^1 | w |^p F_\alpha |\alpha'|_g \,dt}
\]
Since $w$ is arbitrary, this implies that $\lambda_{1,p}(\beta) \le \lambda_{1,p}(\alpha)$.
\end{proof}

Let $\cR_g^*$ be the set of curves $\alpha$ in $\cR_g$ which satisfy the following properties.
First $\alpha$ is simple and $|\alpha'(t)|_g=L_g(\alpha)$ for almost every $t$ in $[0,1]$.
Second $u \circ \alpha(1) >0$.
Third there is a constant $c>0$ such that for all $t$ in $[0,1]$,
\[
\label{rgtrans}
	v \circ \alpha(t) \ge c (1-t)
\]
We can now establish existence of an eigenvalue maximizing curve in $\cR_g^*$.

\begin{Lemma}
\label{gex}
Fix $p \ge 2n-1$.
Assume either $x_0 \neq y_0$ or $\Lambda_p > \lambda_{1,p}(\sigma)$.
Then there is a curve $\alpha$ in $\cR_g^*$ such that $\lambda_{1,p}(\alpha) = \Lambda_p$.
\end{Lemma}

\begin{proof}
By Lemma \ref{rumono}, there is a curve $\beta$ in $\cR_h^+$ such that $\lambda_{1,p}(\beta) = \Lambda_p$.
Moreover $u \circ \beta$ is monotonically increasing and $r \circ \beta$ is monotonically decreasing.
We claim that $r \circ \beta(1) >0$.
To prove this, suppose that $r \circ \beta(1)=0$.
Then $u \circ \beta$ is identically zero and the reparametrization of $\beta$ given by Lemma \ref{garc} is $\sigma$.
In particular $x_0=y_0$ and Lemma \ref{garc} implies that $\lambda_{1,p}(\sigma)=\Lambda_p$.
By this contradiction $r \circ \beta(1)>0$, so $u \circ \beta(1)>0$.
Since $u \circ \beta(0)\ge0$, the monotonicity of $r \circ \beta$ and $u \circ \beta$ together imply that $v \circ \beta$ is monotonically decreasing over $[0,1]$.
The monotonicity of $u \circ \beta$ and $v \circ \beta$ together imply that $\beta$ has finite length with respect to the metric $du^2 + dv^2$.
Since $r \circ \beta$ is positive over $[0,1]$, this implies that $\beta$ has finite length with respect to $g$.
Let $\alpha$ be the reparametrization of $\beta$ given by Lemma \ref{garc}.
Then $\alpha$ is in $\cR_g$ and $| \alpha'(t)|_g=L_g(\alpha)$ for almost every $t$ in $[0,1]$.
Also $\alpha$ is simple, because $u \circ \alpha$ and $v \circ \alpha$ are monotonic.
Furthermore $\lambda_{1,p}(\alpha) \ge \lambda_{1,p}(\beta)$, so $\lambda_{1,p}(\alpha) = \Lambda_p$.
Moreover $r \circ \alpha$ is monotonically decreasing over $[0,1]$ and $u \circ \alpha$ is monotonically increasing over $[0,1]$.
Therefore there is a constant $c>0$ such that if $t$ close to $1$ and $\alpha$ is differentiable at $t$, then
\[
	(v \circ \alpha)'(t) < -c
\]
This implies \eqref{rgtrans}, so $\alpha$ is in $\cR_g^*$.
\end{proof}

\section{Regularity}

In this section we complete the proof of Lemma \ref{exreg} by establishing regularity of a maximizing curve in $\cR_g^*$.
The following lemma gives sufficient conditions for a curve in $\cR_g$ to admit an eigenfunction.
Let $\Lip_{0,\operatorname{loc}}([0,1))$ be the set of functions $w:[0,1) \to \R$ such that $w(0)=0$ and $w$ is locally Lipschitz over $[0,1)$, i.e. Lipschitz over $[0,a]$ for every $a$ in $(0,1)$.

\begin{Lemma}
\label{efex}
Let $\alpha$ be a curve in $\cR_g$.
Assume that $u \circ \alpha(1)>0$.
Let $c>0$ and assume that $| \alpha'(t) |_g \ge c$ for almost every $t$ in $[0,1]$.
Assume that for all $t$ in $[0,1]$,
\[
\label{efexva1}
	v \circ \alpha(t) \ge c(1-t)
\]
Fix $p \ge 2$.
Then there is a function $\phi$ in $\Lip_{0,\operatorname{loc}}([0,1))$ such that
\[
	\frac{\int_0^1 \frac{|\phi'|^p F_\alpha}{| \alpha' |_g^{p-1}} \,dt}{\int_0^1 |\phi|^p F_\alpha | \alpha' |_g \,dt} = \lambda_{1,p}(\alpha)
\]
Moreover $\phi(t) > 0$ for every $t$ in $(0,1)$.
\end{Lemma}

\begin{proof}
Let $F_\alpha=F \circ \alpha$.
Let $L^p(\alpha)$ be the set of measurable functions $f:[0,1] \to \R$ such that
\[
	\| f \|_{L^p(\alpha)} = \bigg( \int_0^1 |f|^p F_\alpha | \alpha'(t) |_g \,dt \bigg)^{1/p} < \infty
\]
Let $C_0^1([0,1])$ be the set of continuously differentiable functions $f:[0,1] \to \R$ such that $f(0)=0$.
Let $W_0^{1,p}(\alpha)$ be the completion of $C_0^1([0,1])$ with respect to the norm
\[
	\| f \|_{W_0^{1,p}(\alpha)} = \bigg( \int_0^1 \frac{|f'|^p F_\alpha}{| \alpha'(t) |_g^{p-1}} \,dt \bigg)^{1/p} < \infty
\]
Note that $\alpha$ is Lipschitz, $u \circ \alpha(1) >0$, and $v \circ \alpha(1)=0$.
Therefore it follows from \eqref{guv} that there is a constant $C>0$ such that
\[
\label{efexva2}
	v \circ \alpha(t) \le C (1-t)
\]
By \eqref{efexva1} and \eqref{efexva2}, there are positive constants $C_1 < C_2$ such that for all $t$ in $[0,1]$,
\[
	C_1 (1-t)^{n-1} \le F_\alpha(t) \le C_2 (1-t)^{n-1}
\]
Let $B_n$ be a unit ball in $\R^n$.
Identify a function $f$ in $L^p(\alpha)$ or $W_0^{1,p}(\alpha)$ with a radial function $w: B_n \to \R$ defined by
\[
	w(x) = f(1-|x|)
\]
The space $L^p(\alpha)$ is a Banach space, equivalent to the subspace of $L^p(B_n)$ consisting of radial functions.
Similarly, the space $W_0^{1,p}(\alpha)$ is a Banach space, equivalent to the subspace of $W_0^{1,p}(B_n)$ consisting of radial functions.
In particular $W_0^{1,p}(\alpha)$ is reflexive.
Note that a function in $W_0^{1,p}(\alpha)$ is necessarily continuous over $[0,1)$.
By the Rellich-Kondrachov theorem, the space $W_0^{1,p}(B_n)$ is compactly embedded in $L^p(B_n)$.
Therefore $W_0^{1,p}(\alpha)$ is compactly embedded in $L^p(\alpha)$.
Now the direct method in the calculus of variations establishes the existence of a function $\phi$ in $W_0^{1,p}(\alpha)$ such that
\[
	\frac{\int_0^1 \frac{|\phi'|^p F_\alpha}{| \alpha' |_g^{p-1}} \,dt}{\int_0^1 |\phi|^p F_\alpha | \alpha' |_g \,dt} = \lambda_{1,p}(\alpha)
\]
We may assume $\phi(t) \ge 0$ for all $t$ in $[0,1]$, by possibly replacing $\phi$ with $|\phi|$.
Moreover $\phi$ weakly satisfies the corresponding Euler-Lagrange equation, i.e.
\[
	-\bigg( \frac{|\phi'|^{p-2} \phi' F_\alpha}{| \alpha' |_g^{p-1}} \bigg)' = \lambda_{1,p}(\alpha) F_\alpha | \alpha'|_g (\phi)^{p-1}
\]
This equation implies that $\phi$ is in $\Lip_{0,\operatorname{loc}}([0,1))$.
Furthermore a Harnack inequality of Trudinger \cite[Theorem 1.1]{NT} implies that $\phi$ does not vanish in $(0,1)$.
\end{proof}

In the next lemma we show that for any function $w$ in $\Lip_{0,\operatorname{loc}}([0,1))$, the Rayleigh quotient of $w$ is greater than or equal to $\lambda_{1,p}$.

\begin{Lemma}
\label{liplemma}
Let $\alpha$ be a curve in $\cR_g$ and let $F_\alpha = F \circ \alpha$.
Let $w$ be in $\Lip_{0,\operatorname{loc}}([0,1))$.
Fix $p \ge 2$ and assume that
\[
	\int_0^1 |w|^p F_\alpha | \alpha' |_g \,dt < \infty
\]
Then
\[
	\lambda_{1,p}(\alpha) \le \frac{\int_0^1 \frac{|w'|^p F_\alpha}{| \alpha' |_g^{p-1}} \,dt}{\int_0^1 |w|^p F_\alpha | \alpha' |_g \,dt}
\]
\end{Lemma}

\begin{proof}
For each $s$ in $(0,1)$, define a function $w_s$ in $\Lip_0([0,1])$ by
\[
	w_s(t) =
	\begin{cases}
		w(t) & t \in [0,s] \\
		w(s) & t \in [s,1]	 \\
	\end{cases}
\]
For each $s$,
\[
	\lambda_{1,p}(\alpha) \le \frac{\int_0^1 \frac{|w_s'|^p F_\alpha}{| \alpha' |_g^{p-1}} \,dt}{\int_0^1 |w_s|^p F_\alpha | \alpha' |_g \,dt}
\]
Applying the monotone convergence theorem and Fatou's lemma,
\[
	\lambda_{1,p}(\alpha) \le \limsup_{s \nearrow 1} \frac{\int_0^1 \frac{|w_s'|^p F_\alpha}{| \alpha' |_g^{p-1}} \,dt}{\int_0^1 |w_s|^p F_\alpha | \alpha' |_g \,dt}
		\le \frac{\int_0^1 \frac{|w'|^p F_\alpha}{| \alpha' |_g^{p-1}} \,dt}{\int_0^1 |w|^p F_\alpha | \alpha' |_g \,dt}
\]
\end{proof}

In the following lemma we show that if a maximizing curve intersects a small circle at two points, then it must stay inside the circle between those points.

\begin{Lemma}
\label{circlesinterior}
Fix $p \ge 2n-1$.
Let $\alpha$ be a curve in $\cR_g^*$ such that $\lambda_{1,p}(\alpha) = \Lambda_p$.
Let $C$ be a large positive constant.
Let $(x_1,y_1)$ be in the interior of $\R^{2n}/G$.
Let $r_1$ be a positive number such that
\[
	C r_1 \le \min(x_1, y_1)
\]
Define
\[
	D = \bigg\{ (x,y) \in \R^{2n}/G: (x-x_1)^2+(y-y_1)^2 \le r_1^2 \bigg\}
\]
Let $0 \le t_1 < t_2 \le 1$ and assume $\alpha(t_1)$ and $\alpha(t_2)$ lie on the boundary $\del D$.
Assume that, for all $t$ in $[t_1, t_2]$,
\[
	| \alpha(t) - (x_1,y_1) | < 2 r_1
\]
If $C$ is sufficiently large, independent of $x_1$, $y_1$, and $r_1$, then it follows that $\alpha(t)$ is in $D$ for all $t$ in $[t_1, t_2]$. 
\end{Lemma}

\begin{proof}
Suppose not.
It suffices to consider the case where $\alpha(t)$ lies outside of $D$ for every $t$ in $(t_1,t_2)$.
There are Lipschitz functions $r:[t_1,t_2] \to (0, \infty)$ and $\theta:[t_1,t_2] \to \R$ such that for all $t$ in $[t_1,t_2]$,
\[
	\alpha(t) = \Big( x_1 + r(t) \cos \theta(t), y_1 + r(t) \sin \theta(t) \Big)
\]
Note that $r_1 < r(t) < 2r_1$ for all $t$ in $(t_1, t_2)$.
Define a curve $\beta$ in $\cR_g^*$ by
\[
	\beta(t) =
	\begin{cases}
		\alpha(t) & t \in [0,t_1) \cup (t_2,1] \\
		\Big( x_1 + \frac{r_1^2}{r(t)} \cos \theta(t), y_1 + \frac{r_1^2}{r(t)} \sin \theta(t) \Big) & t \in [t_1, t_2] \\
	\end{cases}
\]
Let $F_\alpha=F \circ \alpha$ and $F_\beta= F \circ \beta$.
For all $p\ge 2n-1$ and all $t$ in $(t_1,t_2)$,
\[
	\frac{F_\beta(t)}{| \beta'(t) |_g^{p-1}} > \frac{F_\alpha(t)}{| \alpha'(t)|_g^{p-1}}
\]
Also, for all $t$ in $(t_1,t_2)$,
\[
	F_\beta(t) | \beta'(t)|_g < F_\alpha(t) | \alpha'(t)|_g
\]
By Lemma \ref{efex}, there is a function $\phi$ in $\Lip_{0,\operatorname{loc}}([0,1))$ which is non-vanishing over $(0,1)$ and satisfies
\[
	\lambda_{1,p}(\beta) = \frac{ \int_0^1 \frac{| \phi' |^p F_\beta}{| \beta' |_g^{p-1}} \,dt}{\int_0^1 | \phi |^p | \beta' |_g F_\beta \,dt}
\]
Then by Lemma \ref{liplemma},
\[
	\lambda_{1,p}(\alpha) \le \frac{ \int_0^1 \frac{| \phi' |^p F_\alpha}{| \alpha' |_g^{p-1}} \,dt}{\int_0^1 | \phi |^p | \alpha' |_g F_\alpha \,dt} < \frac{ \int_0^1 \frac{| \phi' |^p F_\beta}{| \beta' |_g^{p-1}} \,dt}{\int_0^1 | \phi |^p | \beta' |_g F_\beta \,dt} = \lambda_{1,p}(\beta)
\]
This is a contradiction, because $\lambda_{1,p}(\alpha)=\Lambda$.
\end{proof}

The next lemma is a variation of the previous lemma for circles centered on the boundary of $\R^{2n}/G$.

\begin{Lemma}
\label{circlesendpoint}
Fix $p \ge 2n-1$.
Let $\alpha$ be a curve in $\cR_g^*$ such that $\lambda_{1,p}(\alpha) = \Lambda_p$.
Let $C$ be a large positive constant.
Let $x_1$ be a positive number.
Let $r_1$ be a positive number such that $C r_1 \le x_1$.
Define
\[
	D = \bigg\{ (x,y) \in \R^{2n}/G: (x-x_1)^2+y^2 \le r_1^2 \bigg\}
\]
Let $0 < t_1 < 1$ and assume $\alpha(t_1)$ lies on the boundary $\del D$.
Assume that, for all $t$ in $[t_1, 1]$,
\[
	| \alpha(t) - (x_1,0) | < 2 r_1
\]
If $C$ is sufficiently large, independent of $x_1$ and $r_1$, then it follows that $\alpha(t)$ is in $D$ for all $t$ in $[t_1, 1]$. 
\end{Lemma}

\begin{proof}
Suppose not.
It suffices to consider two cases.
In the first case we assume that there is a number $t_2$ in $(t_1,1]$ such that $\alpha(t)$ lies outside of $D$ for every $t$ in $(t_1,t_2)$ and $\alpha(t_2)$ lies on the boundary $\del D$.
In the second case, we assume that $\alpha(t)$ lies outside of $D$ for every $t$ in $(t_1,1]$.
In the second case, define $t_2=1$.
In either case, define $y_1=0$.
Then repeating the argument used to prove Lemma \ref{circlesinterior} yields a contradiction.
\end{proof}

In the next lemma, we show that an eigenvalue maximizing curve in $\cR_g^*$ is absolutely differentiable.

\begin{Lemma}
\label{abdiff}
Fix $p \ge 2n-1$.
Let $\alpha$ be a curve in $\cR_g^*$ such that $\lambda_{1,p}(\alpha) = \Lambda_p$.
Let $t_0$ be a point in $[0,1)$.
Let $\{ p_k \}$ be a sequence in $[0, t_0]$ converging to $t_0$ and let $\{ q_k \}$ be a sequence in $[t_0,1)$ converging to $t_0$
Assume that $p_k \neq q_k$ for all $k$.
Then
\[
	\lim_{k \to \infty} \frac{| \alpha(q_k)-\alpha(p_k)|}{|q_k-p_k|} = L_g(\alpha)
\]
In particular,
\[
	\lim_{t \to t_0} \frac{| \alpha(t)-\alpha(t_0)|}{|t-t_0|} = L_g(\alpha)
\]
\end{Lemma}

\begin{proof}
Suppose not.
Since $\alpha$ is Lipschitz with constant $L_g(\alpha)$, there is a constant $c$ such that
\[
	\liminf_{k \to \infty} \frac{| \alpha(q_k)-\alpha(p_k)|}{|q_k-p_k|} < c < L_g(\alpha)
\]
By passing to subsequences, we may assume that for all $k$,
\[
	\frac{| \alpha(q_k)-\alpha(p_k)|}{|q_k-p_k|} < c
\]
Fix $k$ large and define a curve $\beta$ in $\cR_g^*$ by
\[
	\beta(t) =
	\begin{cases}
		\alpha(t) & 0\le t \le p_k \\
		\alpha(p_k) + (t-p_k) \frac{\alpha(q_k)-\alpha(p_k)}{q_k-p_k} & p_k \le t \le q_k\\
		\alpha(t) & q_k \le t \le 1
	\end{cases}
\]
Let $F_\beta=F \circ \beta$.
Since $\alpha$ is simple, Lemma \ref{efex} shows that there is a function $\phi$ in $\Lip_{0,\operatorname{loc}}([0,1))$ which is non-vanishing over $(0,1)$ and satisfies
\[
	\lambda_{1,p}(\beta) = \frac{ \int_0^1 \frac{| \phi' |^p F_\beta}{| \beta' |_g^{p-1}} \,dt}{\int_0^1 | \phi |^p | \beta' |_g F_\beta \,dt}
\]
Let $F_\alpha=F \circ \alpha$.
If $k$ is sufficiently large, then for all $t$ in $(p_k, q_k)$,
\[
	\frac{ F_\beta(t)}{| \beta'(t)|_g^{p-1}} > \frac{ F_\alpha(t)}{| \alpha'(t)|_g^{p-1}}
\]
Also for all $t$ in $(p_k,q_k)$,
\[
	 F_\beta(t) | \beta'(t) |_g < F_\alpha(t) | \alpha'(t)|_g
\]
Therefore if $k$ is sufficiently large, then by Lemma \ref{liplemma},
\[
	\lambda_{1,p}(\alpha) \le \frac{ \int_0^1 \frac{| \phi' |^p F_\alpha}{| \alpha' |_g^{p-1}} \,dt}{\int_0^1 | \phi |^p | \alpha' |_g F_\alpha \,dt} < \frac{ \int_0^1 \frac{| \phi' |^p F_\beta}{| \beta' |_g^{p-1}} \,dt}{\int_0^1 | \phi |^p | \beta' |_g F_\beta \,dt} = \lambda_{1,p}(\beta)
\]
This is a contradiction, because $\lambda_{1,p}(\alpha) = \Lambda_p$.
\end{proof}

Now we can show that an eigenvalue maximizing curve in $\cR_g^*$ is differentiable.

\begin{Lemma}
\label{diff}
Fix $p \ge 2n-1$.
Let $\alpha$ be a curve in $\cR_g^*$ such that $\lambda_{1,p}(\alpha) = \Lambda_p$.
Then $\alpha$ is differentiable over $[0,1)$.
Moreover $|\alpha'(t)|_g= L_g(\alpha)$ for every $t$ in $[0,1)$.
\end{Lemma}

\begin{proof}
We first prove that $\alpha$ is right-differentiable over $[0,1)$.
Let $t_0$ be in $[0,1)$ and suppose that $\alpha$ is not right-differentiable at $t_0$.
It follows from Lemma \ref{abdiff} that there is a positive constant $c$ and sequences $\{ y_k \}$ and $\{ z_k \}$ in $(t_0,1)$ converging to $t_0$ such that for all $k$, the points $\alpha(t_0), \alpha(y_k), \alpha(z_k)$ are distinct, and the interior angle at $\alpha(t_0)$ of the triangle with vertices at these points is at least $c$.
By passing to a subsequence we may assume that $y_k < z_k$ for all $k$.
Let $x_\alpha$ and $y_\alpha$ be the component functions of $\alpha$.
Let $C>0$ be a large constant.
Fix a positive constant $r_1$ with
\[
	C r_1 < \min \Big(x_\alpha(t_0), y_\alpha(t_0) \Big)
\]
For large $k$,
\[
	0 < | \alpha(z_k) - \alpha(t_0) | < r_1
\]
Then there are two closed discs of radius $r_1$ which contain $\alpha(z_k)$ and $\alpha(t_0)$ on their boundaries.
If $C$ and $k$ are large, then by Lemma \ref{circlesinterior}, the point $\alpha(y_k)$ must be in the intersection of these discs.
But this implies that the interior angle at $\alpha(t_0)$ of the triangle with vertices at $\alpha(t_0), \alpha(y_k), \alpha(z_k)$ converges to zero as $k$ tends to infinity.
By this contradiction $\alpha$ is right-differentiable over $[0,1)$.

A symmetric argument shows that $\alpha$ is left-differentiable over $(0,1)$.
Then Lemma \ref{abdiff} implies that the left and right derivatives must agree over $(0,1)$ and $| \alpha'(t) |_g = L_g(\alpha)$ for every $t$ in $[0,1)$.
\end{proof}

The following lemma shows that an eigenvalue maximizing curve in $\cR_g^*$ is in $\cC$.

\begin{Lemma}
\label{c1}
Fix $p \ge 2n-1$.
Let $\alpha$ be a curve in $\cR_g^*$ such that $\lambda_{1,p}(\alpha) = \Lambda_p$.
Then $\alpha$ is in $\cC$.
\end{Lemma}

\begin{proof}
Note that $\alpha$ is differentiable over $[0,1)$ and $| \alpha'(t) |_g = L_g(\alpha)$ for every $t$ in $[0,1)$ by Lemma \ref{diff}.
In order to show that $\alpha$ is continuously differentiable over $[0,1)$, fix $t_0$ in $[0,1)$ and let $\{ s_k \}$ be a sequence in $[0,1)$ converging to $t_0$.
Let $x_\alpha$ and $y_\alpha$ be the component functions of $\alpha$.
Let $C>0$ be a large constant.
Let $r_1>0$ be such that
\[
	C r_1 < \min\Big( x_\alpha(t_0), y_\alpha(t_0) \Big)
\]
For large $k$, there are exactly two closed discs in $\R^{2n}/G$ of radius $r_1$ which contain $\alpha(s_k)$ and $\alpha(t_0)$ on their boundaries.
If $k$ is large, then Lemma \ref{circlesinterior} implies that $\alpha(t)$ must lie in the intersection of these discs for all $t$ between $t_0$ and $s_k$.
Since $\alpha$ is differentiable over $[0,1)$, and $| \alpha'(t) |_g = L_g(\alpha)$ for all $t$ in $[0,1)$, it follows that
\[
	\lim_{k \to \infty} | \alpha'(s_k) - \alpha'(t_0) | = 0
\]
Therefore $\alpha'$ is continuous at $t_0$.
This proves that $\alpha$ is continuously differentiable over $[0,1)$.

Fix $t_1$ in $[0,1)$.
Let $C>0$ be a large constant.
Let $r_2>0$ be such that $C r_2 < x_\alpha(t_1)$.
If $t_1$ is close to $1$, then there are exactly two closed half-discs in $\R^{2n}/G$ of radius $r_1$ which are centered on the boundary of $\R^{2n}/G$ and contain $\alpha(t_1)$ on their boundaries.
Lemma \ref{circlesendpoint} implies that $\alpha(t)$ must lie in the intersection of these discs for all $t$ between $t_1$ and $1$.
Since $\alpha$ is differentiable over $[0,1]$, and $| \alpha'(t) |_g = L_g(\alpha)$ for all $t$ in $[0,1]$, it follows that
\[
	\lim_{t \to 1} \alpha'(t) = \Big(0, -L_g(\alpha) \Big)
\]
This implies that $\alpha$ is continuously differentiable over $[0,1]$ and $\alpha'(1)=(0,-L_g(\alpha))$.
Therefore $\alpha$ is in $\cC$.
\end{proof}

We can now prove Lemma \ref{exreg}.

\begin{proof}[Proof of Lemma 2.1]
If $x_0 \neq y_0$ or $\Lambda_p > \lambda_{1,p}(\sigma)$, then by Lemma \ref{gex}, there is a curve $\alpha$ in $\cR_g^*$ such that $\lambda_{1,p}(\alpha) = \Lambda_p$.
In particular $\alpha$ is simple.
Moreover $\alpha$ is in $\cC$ by Lemma \ref{c1}.
\end{proof}

\section{Simons' cone}

In this section we conclude the article by proving Lemma \ref{lemnocone}.
By a scaling argument, it suffices to consider the case $x_0=y_0=1$.
We assume that $x_0=y_0=1$ throughout this section.
Define a function $\sigma_0: [0,1] \to \R^{2n}/G$ by
\[
\label{sigma0}
	\sigma_0(t) = (1-t)^{1/2} \cdot (1,1)
\]
Note that
\[
	\lambda_{1,p}(\sigma_0) = \inf \left \{ \frac{ 2^{p/2} \int_0^1 | w' |^p (1-t)^{n+\frac{p}{2}-\frac{3}{2}} \,dt}{\int_0^1 | w |^p (1-t)^{n-\frac{3}{2}} \,dt} : w \in \Lip_0([0,1]) \right \}
\]
The proof of Lemma \ref{liplemma} shows that it is equivalent to take the infimum over functions $w$ in $\Lip_{0,\operatorname{loc}}([0,1))$ such that $\int_0^1 | w |^p (1-t)^{n-\frac{3}{2}} \,dt$ is finite.
Also $\lambda_{1,p}(\sigma) = \lambda_{1,p}(\sigma_0)$.
To verify this, let $w$ be a function in $\Lip_{0,\operatorname{loc}}([0,1))$ and let $z$ be a function in $\Lip_0([0,1])$.
Assume that $z(t)=w(1-(1-t)^2)$ for all $t$.
Then
\[
	\frac{ 2^{p/2} \int_0^1 | w'(t) |^p (1-t)^{n+\frac{p}{2}-\frac{3}{2}} \,dt}{\int_0^1 | w(t) |^p (1-t)^{n-\frac{3}{2}} \,dt} = \frac{2^{-p/2} \int_0^1 | z'(t) |^p (1-t)^{2n-2} \,dt}{\int_0^1 |z(t)|^p (1-t)^{2n-2} \,dt}
\]
This implies that $\lambda_{1,p}(\sigma) = \lambda_{1,p}(\sigma_0)$.

Recall the coordinate functions $u$ and $v$, defined in \eqref{ucoord} and \eqref{vcoord}.
For all $t$ in $[0,1]$,
\[
	\Big( u \circ \sigma_0(t), v \circ \sigma_0(t) \Big) = (0, 1-t)
\]
For $\nu$ in $\R$, let $J_\nu$ denote the Bessel function of the first kind of order $\nu$.
Let $j_{\nu,1}$ denote the first positive root of $J_\nu$.
Define a function $\phi_\sigma$ in $\Lip_0([0,1])$ by
\[
	\phi_\sigma(t) = (1-t)^{\frac{3-2n}{4}} J_{n-\frac{3}{2}} ( j_{n-\frac{3}{2},1} \sqrt{1-t})
\]
In the following lemma, we express the eigenvalues $\lambda_{1,p}(\sigma)$ in terms of the eigenvalues of a unit ball in $\R^{2n-1}$.

\begin{Lemma}
\label{conesol}
Let $B_{2n-1}$ be the unit ball in $\R^{2n-1}$.
For all $p$,
\[
\label{conesol1}
	\lambda_{1,p}(\sigma) = 2^{-p/2} \lambda_{1,p}(B_{2n-1})
\]
For the case $p=2$,
\[
\label{conesol2}
	\lambda_{1,2}(\sigma) = \frac{ 2 \int_0^1 | \phi_\sigma' |^2 (1-t)^{n-\frac{1}{2}} \,dt}{\int_0^1 | \phi_\sigma |^2 (1-t)^{n-\frac{3}{2}} \,dt} = \frac{1}{2} \cdot j_{n-\frac{3}{2},1}^2
\]
\end{Lemma}

\begin{proof}
Let $w$ be a function in $\Lip_{0,\operatorname{loc}}([0,1))$ and let $v$ be a radial function in $\Lip_0(B_{2n-1})$.
Assume that for all $x$ in $B_{2n-1}$,
\[
	v(x)=w\Big((1-|x|)^2\Big)
\]
For all $p$,
\[
\label{conesol3}
	\frac{ 2^{p/2} \int_0^1 | w' |^p (1-t)^{n+\frac{p}{2}-\frac{3}{2}} \,dt}{\int_0^1 | w |^p (1-t)^{n-\frac{3}{2}} \,dt} = 2^{-\frac{p}{2}} \cdot \frac{\int_B | \grad v |^p}{\int_B |v|^p}
\]
Therefore \eqref{conesol1} follows.
For the case $p=2$, it is a classical fact that $\lambda_{1,2}(B_{2n-1}) = j_{n-\frac{3}{2},1}^2$ and that the Rayleigh quotients in \eqref{conesol3} are minimized when
\[
	v(x) = |x|^{\frac{3}{2}-n} J_{n-\frac{3}{2}}(j_{n-\frac{3}{2},1} |x| )
\]
That is, the quotient is minimized when $w=\phi_\sigma$.
Therefore \eqref{conesol2} holds.
\end{proof}

In the next lemma, we show that if $n$ is fixed and $p$ is large, then $\lambda_{1,p}(\sigma) > \Lambda_p$.
The proof is a simple application of a result of Juutinen, Lindqvist, and Manfredi \cite[Lemma 1.5]{JLM}.
They showed that if $\Omega$ is a smoothly bounded domain in $\R^d$, and if $\operatorname{inrad}(\Omega)$ is the inradius of $\Omega$, then
\[
	\lim_{p \to \infty} \Big( \lambda_{1,p}(\Omega) \Big)^{1/p} = \frac{1}{\operatorname{inrad}(\Omega)}
\]
In particular, if $d \ge 1$ is an integer and $B_d$ is a unit ball in $\R^d$, then
\[
\label{plim}
	\lim_{p \to \infty} \Big( \lambda_{1,p}(B_d) \Big)^{1/p} = 1
\]

\begin{Lemma}
\label{5bigp}
Fix $n \ge 2$.
If $p$ is large, then there is a curve $\alpha$ in $\cC$ such that $\lambda_{1,p}(\alpha) > \lambda_{1,p}(\sigma)$.
\end{Lemma}

\begin{proof}
By \eqref{plim} and Lemma \ref{conesol},
\[
\label{sigmalim}
	\lim_{p \to \infty} \Big( \lambda_{1,p}(\sigma) \Big)^{1/p} = 2^{-1/2}
\]
Define a curve $\alpha$ in $\cC$ by $\alpha(t) = (1,1-t)$.
Let $B_n$ be a unit ball in $\R^n$. 
The hypersurface in $\R^{2n}$ corresponding to $\alpha$ is isometric to $B_n \times \S^{n-1}$.
Moreover $\lambda_{1,p}(B_n \times \S^{n-1}) = \lambda_{1,p}(B_n)$, because the simplicity of the first eigenvalue on $B_n \times \S^{n-1}$ implies that the first eigenfunction is invariant under symmetries.
Simplicity of the first eigenvalue was established by Lindqvist~\cite{Li}.
In particular, $\lambda_{1,p}(\alpha) = \lambda_{1,p}(B_n)$.
Therefore, by \eqref{plim},
\[
\label{alphalim}
	\lim_{p \to \infty} \Big( \lambda_{1,p}(\alpha) \Big)^{1/p} = 1
\]
Now \eqref{sigmalim} and \eqref{alphalim} imply that $\lambda_{1,p}(\sigma) < \lambda_{1,p}(\alpha)$, for large $p$.
\end{proof}

We use a variational argument for the case $p=2$. 
For each $s > 0$, define a curve $\sigma_s$ in $\cC$ such that for all $t$ in $[0,1]$,
\[
	\Big( u \circ \sigma_s(t) , v \circ \sigma_s(t) \Big) = (st, 1-t)
\]
Note that for the case $s=0$, we recover the curve $\sigma_0$ defined in \eqref{sigma0}.
For each $s \ge 0$ define functions $P_s:[0,1) \to \R$ and $Q_s:[0,1) \to \R$ by
\[
	P_s(t) = \frac{2 (1-t)^{n-1}((1-t)^2 + s^2 t^2)^{1/4}}{(1+s^2)^{1/2}}
\]
and
\[
	Q_s(t) = \frac{(1-t)^{n-1} (1+s^2)^{1/2}}{((1-t)^2+s^2t^2)^{1/4}}
\]
For each $s \ge 0$,
\[
	\lambda_{1,2}(\sigma_s) = \min \left \{ \frac{\int_0^1 |w'(t)|^2 P_s(t) \,dt}{\int_0^1 |w(t)|^2 Q_s(t) \,dt} : w \in \Lip_0([0,1]) \right \}
\]
For each $s>0$, let $\phi_s$ be the eigenfunction in $\Lip_{0,\operatorname{loc}}([0,1))$ corresponding to $\lambda_{1,2}(\sigma_s)$, given by Lemma \ref{efex}.
Let $\phi_0$ be a scalar multiple of $\phi_\sigma$.
Then for each $s \ge 0$,
\[
\label{phismin}
	\lambda_{1,2}(\sigma_s) = \frac{ \int_0^1 P_s |\phi_s'|^2 \,dt}{\int_0^1 Q_s |\phi_s|^2 \,dt}
\]
The eigenfunction $\phi_s$ satisfies the associated Euler-Lagrange equation, i.e.
\[
\label{phiseveq}
	-(P_s \phi_s')' = \lambda_{1,2}(\sigma_s) Q_s \phi_s
\]
This equation implies that $\phi_s$ is twice continuously differentiable over $[0,1)$.
Moreover $\phi_s'(0)$ is non-zero.
For each $s \ge 0$, normalize $\phi_s$ so that $\phi_s'(0)=1$.

In the next lemma, we show that $\lambda_{1,2}(\sigma_s)$ depends continuously on $s$.
Note that if $s_1 \ge 0$ and $s_2 >0$, then for all $t$ in $[0,1)$,
\[
\label{pqbounds}
	\frac{P_{s_1}(t)}{P_{s_2}(t)} = \frac{Q_{s_2}(t)}{Q_{s_1}(t)} \le \bigg( \frac{1+s_2^2}{1+s_1^2} \bigg)^{1/2} \cdot \max \bigg(1, \frac{s_1^{1/2}}{s_2^{1/2}} \bigg)
\]

\begin{Lemma}
\label{coneevcont}
The function $s \mapsto \lambda_{1,2}(\sigma_s)$ is continuous over $[0, \infty)$.
\end{Lemma}

\begin{proof}
Fix $s_0 \ge 0$.
The bounds in \eqref{pqbounds} imply that, for all $s>0$,
\[
	\lambda_{1,2}(\sigma_{s_0}) \le \lambda_{1,2}(\sigma_s) \cdot \frac{1+s^2}{1+s_0^2} \cdot \max \Big(1, \frac{s_0}{s} \Big)
\]
In particular,
\[
\label{coneevcont2}
	\lambda_{1,2}(\sigma_{s_0}) \le \liminf_{s \to s_0} \lambda_{1,2}(\sigma_s)
\]
Note that for all $s \ge 0$,
\[
	\lambda_{1,2}(\sigma_s) \le \frac{\int_0^1 | \phi_{s_0}' |^2 P_s(t) \,dt}{\int_0^1 | \phi_{s_0} |^2 Q_s(t) \,dt}
\]
Additionally,
\[
\label{coneevcont3}
	\lim_{s \to s_0} \frac{\int_0^1 | \phi_{s_0}' |^2 P_s(t) \,dt}{\int_0^1 | \phi_{s_0} |^2 Q_s(t) \,dt}
	= \frac{\int_0^1 | \phi_{s_0}' |^2 P_{s_0}(t) \,dt}{\int_0^1 | \phi_{s_0} |^2 Q_{s_0}(t) \,dt}
	= \lambda_{1,2}(\sigma_{s_0})
\]
If $s_0>0$, then \eqref{coneevcont3} follows from \eqref{pqbounds}.
For the case $s_0=0$, note that $\phi_0$ and $\phi_0'$ are bounded and $P_s$ and $Q_s$ are uniformly bounded for small $s$, because $n \ge 2$.
Then \eqref{coneevcont3} follows from Lebesgue's dominated convergence theorem.
Now
\[
\label{coneevcont1}
	\limsup_{s \to s_0} \lambda_{1,2}(\sigma_s) \le \lambda_{1,2}(\sigma_{s_0})
\]
By \eqref{coneevcont2} and \eqref{coneevcont1}, the function $s \mapsto \lambda_{1,2}(\sigma_s)$ is continuous at $s=s_0$.
\end{proof}

Next we show that the eigenfunctions $\phi_s$ converge to $\phi_\sigma$.

\begin{Lemma}
\label{coneefbound}
For all $t$ in $[0,1)$,
\[
	\lim_{s \to 0} \phi_s(t) = \phi_\sigma(t)
\]
and
\[
	\lim_{s \to 0} \phi'_s(t) = \phi'_\sigma(t)
\]
For any $\delta>0$, the convergence in both limits is uniform over $[0,1-\delta]$.
\end{Lemma}

\begin{proof}
Note that the eigenfunctions $\phi_s$ satisfy \eqref{phiseveq}.
Moreover $\phi_s(0) = 0$, and the functions $\phi_s$ are normalized so that $\phi_s'(0)=1$.
The convergence now follows from Lemma \ref{coneevcont} and continuous dependence on parameters.
\end{proof}

Define $D_-: (0,\infty) \to \R$ to be the lower left Dini derivative of the function $s \mapsto \lambda_{1,2}(\sigma_s)$.
That is, for each $s_0$ in $(0,\infty)$,
\[
	D_-(s_0) = \liminf_{s \nearrow s_0} \frac{\lambda_{1,2}(\sigma_s)-\lambda_{1,2}(\sigma_{s_0})}{s-s_0}
\]
For each $s>0$, define functions $\dot P_s:[0,1] \to \R$ and $\dot Q_s:[0,1] \to \R$ by
\[
	\dot P_s(t) = \frac{st^2 (1-t)^{n-1}}{((1-t)^2+s^2t^2)^{3/4}(1+s^2)^{1/2}} - \frac{2s(1-t)^{n-1}((1-t)^2+s^2t^2)^{1/4}}{(1+s^2)^{3/2}}
\]
and
\[
	\dot Q_s(t) = \frac{s (1-t)^{n-1}}{((1-t)^2+s^2t^2)^{1/4}(1+s^2)^{1/2}} - \frac{st^2(1-t)^{n-1}(1+s^2)^{1/2}}{2((1-t)^2+s^2t^2)^{5/4}}
\]
The following lemma establishes a lower bound for $D_-(s)$.

\begin{Lemma}
\label{conedini}
If $s>0$ is small, then
\[
	D_-(s) \ge \frac{\int_0^1 |\phi_s'|^2 \dot P_s - \lambda_{1,2}(\sigma_s) |\phi_s|^2 \dot Q_s \,dt}{\int_0^1 |\phi_s|^2 Q_s \,dt}
\]
\end{Lemma}

\begin{proof}
Let $s_0>0$ be small.
Define a function $h:(0,\infty) \to \R$ by
\[
	h(s) = \frac{\int_0^1 | \phi_{s_0}' |^2 P_s(t) \,dt}{\int_0^1 | \phi_{s_0} |^2 Q_s(t) \,dt}
\]
Note that $\lambda_{1,2}(\sigma_s) \le h(s)$ for every $s>0$ by Lemma \ref{liplemma}, and $\lambda_{1,2}(\sigma_{s_0}) = h(s_0)$.
Therefore
\[
	D_-(s_0) = \liminf_{s \nearrow s_0} \frac{\lambda_{1,2}(\sigma_s)-\lambda_{1,2}(\sigma_{s_0})}{s-s_0} \ge \liminf_{s \nearrow s_0} \frac{h(s)-h(s_0)}{s-s_0}
\] 
It suffices to show that
\[
\label{conedini1}
	\liminf_{s \nearrow s_0} \frac{h(s)-h(s_0)}{s-s_0} \ge \frac{\int_0^1 |\phi_{s_0}'|^2 \dot P_{s_0} - \lambda_{1,2}(\sigma_{s_0}) |\phi_{s_0}|^2 \dot Q_{s_0} \,dt}{\int_0^1 |\phi_{s_0}|^2 Q_{s_0} \,dt}
\]
Note that
\[
\label{conedini2}
	\frac{h(s)-h(s_0)}{s-s_0} = \frac{\int_0^1 |\phi_{s_0}'|^2 \frac{P_s-P_{s_0}}{s-s_0} - \lambda_{1,2}(\sigma_{s_0}) |\phi_{s_0}|^2 \frac{Q_s-Q_{s_0}}{s-s_0} \,dt}{\int_0^1 |\phi_{s_0}|^2 Q_s \,dt} 
\]
By \eqref{pqbounds},
\[
\label{conedini3}
	\limsup_{s \nearrow s_0} \int_0^1 |\phi_{s_0}|^2 Q_s \,dt \le \int_0^1 |\phi_{s_0}|^2 Q_{s_0} \,dt
\]
Since $s_0$ is small, there is a $\delta>0$ such that if $0<s<s_0$, then $P_s < P_{s_0}$ and $Q_s > Q_{s_0}$ over $[1-\delta,1]$.
Therefore
\[
\label{conedini4}
\begin{split}
	\liminf_{s \nearrow s_0} \int_0^1 |\phi_{s_0}'|^2 \frac{P_s-P_{s_0}}{s-s_0} &- \lambda_{1,2}(\sigma_{s_0}) |\phi_{s_0}|^2 \frac{Q_s-Q_{s_0}}{s-s_0} \,dt \\
		&\ge \int_0^1 |\phi_{s_0}'|^2 \dot P_{s_0} - \lambda_{1,2}(\sigma_{s_0}) |\phi_{s_0}|^2 \dot Q_{s_0} \,dt
\end{split}
\]
To prove this inequality, use the uniform convergence of the integrands over $[0,1-\delta]$ and use Fatou's lemma over $[1-\delta,1]$.
Now \eqref{conedini2}, \eqref{conedini3}, and \eqref{conedini4} imply \eqref{conedini1}, completing the proof.
\end{proof}

Define functions $\ddot P_0:[0,1] \to \R$ and $\ddot Q_0:[0,1] \to \R$ by
\[
	\ddot P_0(t) = t^2 (1-t)^{n-\frac{5}{2}} - 2(1-t)^{n-\frac{1}{2}}
\]
and
\[
	\ddot Q_0(t) = (1-t)^{n-\frac{3}{2}} - \frac{t^2(1-t)^{n-\frac{7}{2}}}{2}
\]
The following lemma gives a sufficient condition to verify that $\lambda_{1,2}(\sigma) < \lambda_{1,2}(\sigma_s)$ for small $s>0$.

\begin{Lemma}
\label{conemin}
Fix $n$ and assume that
\[
\label{conemin0}
	\int_0^1 | \phi_\sigma'|^2 \ddot P_0(t) - \lambda_{1,2}(\sigma) | \phi_\sigma |^2 \ddot Q_0(t) \,dt > 0
\]
If $s$ is small and positive, then $\lambda_{1,2}(\sigma_s) > \lambda_{1,2}(\sigma)$.
\end{Lemma}

\begin{proof}
By Lemma \ref{coneevcont}, the function $s \mapsto \lambda_{1,2}(\sigma_s)$ is continuous.
Therefore it suffices to show that the Dini derivative $D_-(s)$ is positive for small positive $s$.
In particular, it suffices to show that
\[
	\liminf_{s \searrow 0} s^{-1} D_-(s) \int_0^1 |\phi_s|^2 Q_s \,dt > 0
\]
By Lemma \ref{conedini}, it suffices to show that
\[
\label{conemin1}
	\liminf_{s \searrow 0} \int_0^1 |\phi_s'|^2 s^{-1} \dot P_s - \lambda_{1,2}(\sigma_s) |\phi_s|^2 s^{-1} \dot Q_s \,dt > 0
\]
Fix a small $\delta>0$.
If $s$ is small and positive, then $\dot P_s \ge 0$ over $[1-\delta,1]$ and $\dot Q_s \le 0$ over $[1-\delta,1]$.
Therefore
\[
\label{conemin2}
\begin{split}
	\liminf_{s \searrow 0} \int_0^1 |\phi_s'|^2 &s^{-1} \dot P_s - \lambda_{1,2}(\sigma_s) |\phi_s|^2 s^{-1} \dot Q_s \,dt \\
		&\ge \int_0^1 | \phi_\sigma'|^2 \ddot P_0(t) - \lambda_{1,2}(\sigma) | \phi_\sigma |^2 \ddot Q_0(t) \,dt
\end{split}
\]
To prove this inequality, use Lemma \ref{coneevcont} and Lemma \ref{coneefbound} to obtain uniform convergence of the integrands over $[0,1-\delta]$ and use Fatou's lemma over $[1-\delta,1]$.
Now \eqref{conemin0} and \eqref{conemin2} imply \eqref{conemin1}, completing the proof.
\end{proof}

Now we verify the condition \eqref{conemin0} for $n \le 5$.

\begin{Lemma}
\label{exp}
If $n \le 5$, then
\[
\label{exp3}
	\int_0^1 | \phi_\sigma'|^2 \ddot P_0(t) - \lambda_{1,2}(\sigma) | \phi_\sigma |^2 \ddot Q_0(t) \,dt > 0
\]
\end{Lemma}

\begin{proof}
Fix $n \ge 2$, and let $\alpha={n-\frac{3}{2}}$.
Using the identity $\frac{\alpha}{x} J_\alpha(x) - J'_\alpha(x)= J_{\alpha+1}(x)$, the derivative $\phi_\sigma'$ can be expressed as
\[
	\phi_\sigma'(t) = \frac{j_{\alpha,1}}{2} (1-t)^{-\frac{\alpha+1}{2}} J_{\alpha+1}(j_{\alpha,1} \sqrt{1-t})
\]
Using Lemma \ref{conesol} and changing variables, we have
\[
\label{exp0}
\begin{split}
	\int_0^1 | \phi_\sigma'|^2 &\ddot P_0(t) - \lambda_{1,2}(\sigma) | \phi_\sigma |^2 \ddot Q_0(t)\,dt \\
	&= \int_0^{j_{\alpha,1}} t \Big( |J_{\alpha+1}(t)|^2 + |J_\alpha(t)|^2 \Big) \bigg(\frac{(j_{\alpha,1}^2-t^2)^2}{2t^4} -1 \bigg) \,dt
\end{split}
\]
This integral can be approximated precisely.
Define a function $f:\R \to \R$ by
\[
	f(t) = \frac{t^2}{2} \Big( J_{\alpha+1}(t)^2-J_{\alpha+2}(t) J_\alpha(t) + J_\alpha(t)^2 - J_{\alpha+1}(t) J_{\alpha-1}(t) \Big)
\]
By Lommel's integral,
\[
	f'(t) = t \Big( |J_{\alpha+1}(t)|^2 + |J_\alpha(t)|^2 \Big)
\]
If $n \le 5$, it follows that
\[
\label{exp1}
	\int_0^{j_{\alpha,1}} t \Big( |J_{\alpha+1}(t)|^2 + |J_\alpha(t)|^2 \Big) \,dt < 4
\]
Define a function $g:(0, \infty) \to \R$ by
\[
	g(t) = \frac{(j_{\alpha,1}^2-t^2)^2}{2t^4}
\]
For any partition $0 = p_0 < p_1 < \ldots < p_m = j_{\alpha,1}$, the monotonicity of $g$ implies
\[
	\int_0^{j_{\alpha,1}} t \Big( |J_{\alpha+1}(t)|^2 + |J_\alpha(t)|^2 \Big) \frac{(j_{\alpha,1}^2-t^2)^2}{2t^4} \,dt > \sum_{i=1}^m \Big( f(p_i)-f(p_{i-1}) \Big) g(p_i)
\]
If $n \le 5$, then choosing a suitable partition shows that
\[
\label{exp2}
	\int_0^{j_{\alpha,1}} t \Big( |J_{\alpha+1}(t)|^2 + |J_\alpha(t)|^2 \Big) \frac{(j_{\alpha,1}^2-t^2)^2}{2t^4} \,dt > 4
\]
Now \eqref{exp0}, \eqref{exp1}, and \eqref{exp2} imply \eqref{exp3}, completing the proof.
\end{proof}

Next we round off a curve $\sigma_s$ with $s>0$ to obtain a curve in $\cC$.

\begin{Lemma}
\label{roundoff}
If $n \le 5$, then there is a curve $\alpha$ in $\cC$ such that $\lambda_{1,2}(\alpha) > \lambda_{1,2}(\sigma)$.
\end{Lemma}

\begin{proof}
Fix $s>0$ small.
Then $\lambda_{1,2}(\sigma_s) > \lambda_{1,2}(\sigma)$ by Lemma \ref{conemin} and Lemma \ref{exp}.
Let $u_s = u \circ \sigma_s$ and $v_s = v \circ \sigma_s$.
Note that $u_s(1)=s$.
Let $L$ be the line segment in $\R^2$ given by
\[
	L = \Big \{ \Big(u_s(t), v_s(t)\Big) : t \in [0,1] \Big\}
\]
Let $\delta$ be a small constant satisfying $0<\delta<s$.
There is a disc $D$ in $\R^2$ which is centered about $(s - \delta,0)$ such that $L \cap D$ consists of exactly one point.
Fix $t_0$ so that $\sigma(t_0)$ is the point in $L \cap D$.
There is a unique continuous function $u_\delta:[0,1] \to [0, \infty)$ which agrees with $u_s$ over $[0,t_0]$ such that $(u_\delta(t), v_s(t))$ is in the boundary $\del D$ for all $t$ in $[t_0,1]$.
Let $\beta$ be the curve in $\cR_g$ such that $u \circ \beta = u_\delta$ and $v \circ \beta = v_s$.
Note that $F \circ \beta = F \circ \sigma_s$.
Let $\epsilon>0$ be small.
If $\delta$ is small, then for all $t$ in $[0,1]$,
\[
	|\beta'(t)|_g \le (1 + \epsilon) |\sigma_s'(t)|_g
\]
This yields $\lambda_{1,2}(\beta) \ge (1 + \epsilon)^{-2} \lambda_{1,2}(\sigma_s)$.
If $\epsilon$ is small, then $\lambda_{1,2}(\beta) > \lambda_{1,2}(\sigma)$.
Let $\alpha$ be the reparametrization of $\beta$ given by Lemma \ref{garc}.
Then $\alpha$ is in $\cC$ and $\lambda_{1,2}(\alpha) > \lambda_{1,2}(\sigma)$.
\end{proof}

We can now prove Lemma 2.2.

\begin{proof}[Proof of Lemma 2.2]
It suffices to consider the case $x_0=y_0=1$, by a scaling argument.
The case where $n$ is fixed and $p$ is large is established by Lemma \ref{5bigp}.
The case where $p=2$ and $n \le 5$ is established by Lemma \ref{roundoff}.
\end{proof}

\bibliography{bib115}
\bibliographystyle{plain}

\end{document}